\documentclass[reqno,12pt]{amsart}

\usepackage[margin=1in]{geometry}
\usepackage[dvipsnames]{xcolor}
\usepackage{graphicx, array}
\usepackage{amssymb, amsmath, amsthm, enumerate, subfigure, numprint}
\usepackage[parfill]{parskip}    
\usepackage{tikz}
\usetikzlibrary{calc, positioning, through, math}

\usepackage{tabularx}

\usepackage{soul}  

\newcommand{\darcr}[3]
{\draw (#1) edge [<-, bend right=20] node[midway, fill=white, inner sep=1 pt,font = \tiny]{$#2$} (#3); 
\draw[ultra thick] (#3)edge [bend right=20]  (#1);}

\definecolor{mycyan}{RGB}{0, 255, 255}
\definecolor{mymagenta}{RGB}{255,0, 255}

\usepackage[color links=true, backref=page]{hyperref}

\DeclareGraphicsExtensions{.pdf}

\theoremstyle{plain}
\newtheorem{theorem}{Theorem}
\newtheorem{prop}[theorem]{Proposition}

\newtheorem{lemma}[theorem]{Lemma}
\newtheorem{corollary}[theorem]{Corollary}
\theoremstyle{definition}

\newtheorem{example}{Example}

\theoremstyle{remark}
\newtheorem{remark}{Remark}
\newcommand{\be}{\begin{enumerate}}
\newcommand{\ee}{\end{enumerate}}

\newcommand{\mc}[1]{\mathcal{#1}}

\newenvironment{matr}[1]{\left( \begin{array}{#1}}{\end{array}\right)}

\title[Geometric $(n_{k})$ configurations exist for all $n$]{ 
Eventually, geometric $(n_{k})$ configurations\\ exist for all $n$}

\author[L.W. Berman]{Leah Wrenn Berman}
\author[G. G\'evay]{G\'{a}bor G\'{e}vay}
\author[T. Pisanski]{Toma\v{z} Pisanski}
\address[1]{Department of Mathematics and Statistics, University of Alaska Fairbanks, Fairbanks, AK, USA}
\address[2]{Bolyai Institute, $\!$University of Szeged, Szeged, $\!$Hungary \newline ORCID ID: https://orcid.org/0000-0002-5469-5165}
\address[3]{University of Primorska, Koper, Slovenia, and 
Institute of Mathematics, Physics and Mechanics, University of Ljubljana, Ljubljana, Slovenia \newline ORCID ID: https://orcid.org/0000-0002-1257-5376 }

\date{\today}							

\begin{document}

\begin{abstract}
In a series of papers and in his 2009 book on configurations  Branko Gr\"unbaum described a sequence of operations 
to produce new $(n_{4})$ configurations from various input configurations. These operations were later called the 
``Gr\"unbaum Incidence Calculus''. We generalize two of these operations to produce operations on arbitrary $(n_{k})$ 
configurations. Using them, we show that for any $k$ there exists an integer $N_k$ such that for any $n \geq N_k$ 
there exists a geometric $(n_k)$ configuration. We use empirical results for $k = 2, 3, 4$, and some more detailed 
analysis to improve the upper bound for larger values of $k$.
\end{abstract}

\maketitle

\textbf{Keywords:} axial affinity, geometric configuration, Gr\"{u}nbaum calculus

\textbf{MSC (2020):} 51A45,  51A20, 05B30, 51E30
\bigskip

\begin{center}
\fontsize{12}{14}{\sc
In memory of Branko Gr\"unbaum}
\end{center}
\bigskip

\section{Introduction}

In a series of papers and in his 2009 book on configurations~\cite{Gru2009b}, Branko Gr\"unbaum described a sequence 
of operations to produce new $(n_{4})$ configurations from various input configurations. These operations were later 
called the ``Gr\"unbaum Incidence Calculus''~\cite[Section 6.5]{PisSer2013}. Some of the operations described by 
Gr\"unbaum are specific to producing 3- or 4-configurations. Other operations can be generalized in a straightforward 
way to produce $(n_{k})$ configurations from either smaller $(m_{k})$ configurations with certain properties, or from 
$(m_{k-1})$ configurations. Let $N_{k}$ be the smallest number such that for any $n, n \geq N_k$ there exists a geometric $(n_k)$ configuration.
For $k = 2$ and $k = 3$, the exact value of $N_{k}$ is known, and for $k = 4$ it is known that $N_{4} = 20$ or $24$. 
We generalize two of the Gr\"unbaum Calculus operations in order to prove that for any integer $k$ there exists an 
integer $N_k$  and we give bounds on $N_{k}$ for $k \geq 5$.

The existence of geometric 2-configurations is easily established. The only (connected) combinatorial configuration $(n_2)$ is an $n$-lateral. 
For each $n, n\geq 3$, an $n$-lateral can be realized as a geometric multilateral (for the definition of a \emph{multilateral}, see~\cite{Gru2009b}). 
As a specific example, an $(n_{2})$ configuration can be realized as a regular $n$-gon with sides that are extended to lines. (For larger values 
of $n$ it can also be realized as an $n$-gonal star-polygon, but the underlying combinatorial structure is the same.) Hence:
\begin{prop}\label{thm:n2}
A geometric $(n_2)$ configuration exists if and only if $n \geq 3$. In other words, $N_2 = 3$.
\end{prop}
For 3-configurations, $N_{3}$ is known to be 9 (see \cite[Section 2.1]{Gru2009b}); for example, Branko Gr\"{u}nbaum provides a proof 
(following that of Schr\"{o}ter from 1888, see the discussion in \cite[p. 65]{Gru2009b}) that the cyclic combinatorial configuration 
$\mc{C}_{3}(n)$, which has starting block $[0,1,3]$, can always be realized with straight lines for any $n \geq 9$. That is:
\begin{prop}\label{thm:n3}
A geometric $(n_3)$ configuration exists if and only if $n \geq 9$. In other words, $N_3 = 9$.
\end{prop}
Note that there exist two combinatorial 3-configurations, 
namely $(7_3)$ and $(8_3)$, that do not admit a geometric realization. 

For $k = 4$, the problem of parameters for the existence of 4-configurations is much more complex, 
and the best bound $N_4$ is still not known. For a number of years, the smallest known 4-configuration 
was the $(21_{4})$ configuration which had been studied combinatorially by Klein and others, and 
whose geometric realization, first shown in 1990 \cite{GruRig1990}, initiated the modern study of 
configurations. In that paper, the authors conjectured that this was the smallest $(n_{4})$ configuration. 
In a series of papers \cite{Gru2000, Gru2000b, Gru2002, Gru2006} (summarized in \cite[Sections 3.1-3.4]{Gru2009b}), 
Gr\"unbaum showed that $N_{4}$ was finite and less than 43. In 2008, Gr\"{u}nbaum found a geometrically 
realizable $(20_{4})$ configuration \cite{Gru2008a}. In 2013, J\"{u}rgen Bokowski and Lars Schewe \cite{BokSch2013} 
showed that geometric $(n_{4})$ configurations exist for all $n \geq 18$ except possibly $n = 19, 22, 23, 26, 37, 43$. 
Subsequently, Bokowski and Pilaud \cite{BokPil2015} showed that there is no geometrically realizable $(19_{4})$ configuration, 
and they found examples of realizable $(37_{4})$ and $(43_{4})$ configurations \cite{BokPil2016}. In 2018, Michael 
Cuntz \cite{Cun2018} found realizations of $(22_{4})$ and $(26_{4})$ configurations. However, the question of whether 
a geometric $(23_{4})$ geometric configuration exists is currently still open.

In this paper, $\bar{N_k}$ will denote any known upper bound for $N_k$ and $N^R_k$ will denote currently best upper bound for $N_k$.

Summarizing the above results, we conclude:
\begin{prop}\label{prop:N-4-results}
A geometric $(n_4)$ configuration exists for $n = 18,20,21,22$ and  $n \geq 24$. Moreover, either $N_{4} = 20$ or $N_{4} = 24$ 
(depending on whether or not a $(23_{4})$ configuration exists). In other words, $N^R_4 = 24.$
 \end{prop}

The main result of the paper is the following result.

\begin{theorem}\label{mainTheorem}
For each integer $k \geq 2$ the numbers $N_k$ exist. 
\end{theorem}

To simplify subsequent discussions, we introduce the notion of \emph{configuration-realizability}, 
abbreviated as \emph{realizability}, of numbers. A number $n$ is \emph{$k$-realizable} if and only if there 
exists a geometric $(n_k)$ configuration. We may rephrase Proposition \ref{prop:N-4-results} 
by stating that the numbers $n = 18, 20, 21, 22$ and $n \geq 24$ are $4$-realizable.  Also note that the 
number $9$ is $2$- and $3$-realizable but not $k$-realizable for any $k \geq 4$.


\section{Generalizing two constructions from the Gr\"{u}nbaum Incidence Calculus} \label{sect:Grunbaum}


In this section, we generalize two constructions of the Gr\"unbaum Incidence Calculus which we will use to prove the existence 
of $N_{k}$ for any $k$. As input to examples of these constructions, we often will use the standard geometric realization of 
the $(9_{3})$ Pappus configuration $\mc{P}$, shown in Figure \ref{fig:pappus}. 

\begin{figure}[htbp]
\begin{center}
\includegraphics[width=.5\textwidth]{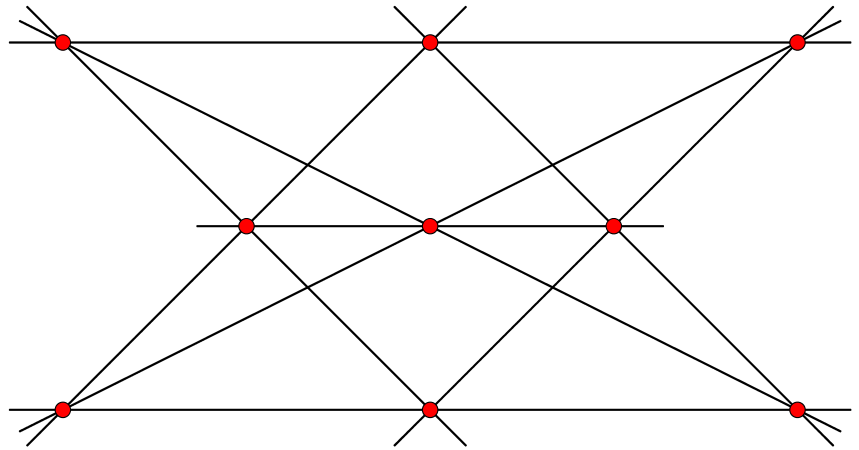}
\caption{The standard geometric realization of the $(9_{3})$ Pappus configuration $\mc{P}$.}
\label{fig:pappus}
\end{center}
\end{figure}

\newcommand{\AR}{\mathrm{AR}}

The first, which we call \emph{affine replication} and denote $\AR(m, k)$, 
generalizes Gr\"unbaum's $\mathbf{(5m)}$ construction; 
it takes as input an $(m_{k-1})$ configuration and produces a $((k+1)m_{k})$ configuration with a pencil of $m$ parallel lines. 

\newcommand{\AS}{\mathrm{AS+}}

The second, which we call 
\emph{affine switch}, 
is analogous to Gr\"unbaum's $\mathbf{(3m+)}$ construction. 
It takes as input a single $(m_{k})$ configuration with a set of $p$ parallel lines in one direction and a set of $q$ parallel lines in 
a second direction which are disjoint (in terms of configuration points) from the pencil of $p$ lines, 
and it produces a configuration $((k-1)m+r)_{k})$ for any $r$ with $1 \leq r \leq p+q$. Applying a series of affine switches to a single 
starting $(m_{k})$ configuration with a pencil of $q$ parallel lines produces a consecutive sequence (or ``band'') of configurations 
\[((k-1)m+1)_{k}), \ldots, ((k-1)m+q)_{k})\]
which we will refer to as $\AS(m,k,q)$.

\subsection{Affine Replication
} 

Starting from an $(m_{k-1})$ configuration $\mc{C}$ we construct a new configuration $\mc{D}$ which is a $((k+1)m_{k})$ configuration. 
A sketch of the construction is that $k-1$ affine images of $\mc{C}$ are carefully constructed so that each point $P$ of $\mc{C}$ is collinear 
with the $k-1$ images of $P$, and each line of $\mc{C}$ and its images are concurrent at a single point. Then $\mc{D}$ consists of the the 
points and lines of $\mc{C}$ and its images, the new lines corresponding to the collinearities from each point $P$, and the new points of 
concurrence corresponding to the lines of $\mc{C}$ and their images.

The details of the construction are as follows:

\be

\item Let $A$ be a line that (i) does not pass through the intersection of two lines of $\mc{C}$, whether or not that intersection point 
is a point of the configuration; (ii) is perpendicular to no line connecting any two points of $\mc{C}$, whether or not that line is a line 
of the configuration; (iii) intersects all lines of $\mc{C}$.

\item 
Let $\alpha_{1}, \alpha_{2}, \ldots \alpha_{k-1}$ be pairwise different orthogonal axial affinities with axis $A$. 
Construct copies $\mc{C}_{1} = \alpha_{1}(\mc{C})$, $\mc{C}_{2} = \alpha_{2}(\mc{C})$,\ldots, 
$\mc{C}_{k-1} = \alpha_{k-1}(\mc{C})$ of $\mc{\mc{C}} = \mc{C}_{0}$.

\item 
Let $\ell$ be any line of $\mc{C}$. Since $A$ is the common axis of each $\alpha_i$, the point $A\cap \ell$ is fixed by all these affinities.
This means that the $k$-tuple of lines $\ell, \alpha_{1}(\ell), \ldots, \alpha_{k-1}(\ell)$ has a common point of intersection
lying on $A$. We denote this point by $F_{\ell}$. By condition (i) in (1), for different lines $\ell, \ell' \in \mc {C}$ the points 
$F_{\ell}, F_{\ell'}$ differ from each other; they also differ from each point of the configurations $\mc {C}_i $ $(i=1,2,\dots, k-1)$.  
We denote the set $\{F_{\ell}: \ell\in \mc{C}\}$ of points lying on $A$ by $\mc{F}$.

\item
Let $P$ be any point of $\mc{C}$. Since the affinities $\alpha_{i}$ are all orthogonal affinities (with the common axis $A$), 
the $k$-tuple of points $P, \alpha_{1}(P), \dots, \alpha_{k-1}(P)$ lies on a line perpendicular to $A$ (and avoids $A$, by condition (i)).
We denote this line by $\ell_P$. Clearly, we have altogether $m$ such lines, one for each point of $\mc{C}$, with no two of them coinciding, by condition (ii).
We denote this set $\{\ell_P: P\in \mc{C}\}$ of lines by $\mc L$.

\item
Put $\mc{D} = \mc{C}_0\cup \mc{C}_{1}\cup\dots \cup\mc{C}_{k-1}\cup \mc{F} \cup \mc L$.

\ee
\begin{figure}[h]
\begin{center}
\includegraphics[width=.66\textwidth]{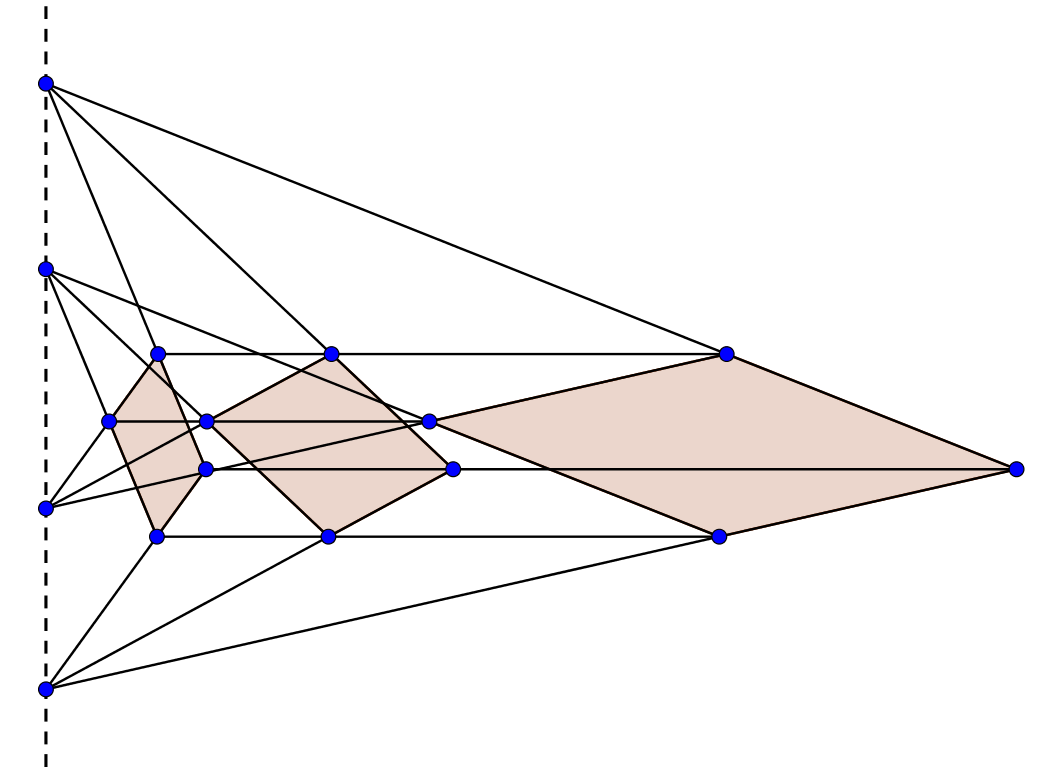}
\caption{Affine replication $\AR(4, 3)$ applied to a quadrilateral, i.e.\ a $(4_2)$ configuration;
it results in a $(16_3)$ configuration. The corresponding ordinary quadrangles are shaded (the starting, 
hence each of the three quadrangles are parallelograms). The axis $A$ is shown by a dashed line.}
\label{fig:Quadrilaterals}
\end{center}
\end{figure}

\noindent
The conditions of the construction imply that $\mc{D}$ is a $((k+1)m_{k})$ configuration. 
Moreover, by construction, $\mc{D}$ has a pencil of $m$ parallel lines. Figures \ref{fig:Quadrilaterals} and \ref{fig:PappusExt} 
show two examples of affine replication, first starting with a $(4_{2})$ configuration to produce a $(16_{3})$ configuration, and 
then starting with the $(9_{3})$ Pappus configuration to produce a $(45_{4})$ configuration.

\begin{remark}
The orthogonal affinities used in the construction are just a particular case of the axial affinities called \emph{strains}~\cite{Cox1969}; 
they can be replaced by other types of axial affinities, namely, by oblique affinities (each with the same (oblique) direction), and even, 
by \emph{shears} (where the direction of affinity is parallel with the axis)~\cite{Cox1969}, while suitably adjusting conditions (i--iii) in (1).
\end{remark}

\begin{figure}[h!]
\begin{center}
\includegraphics[width=0.55\textwidth  
]{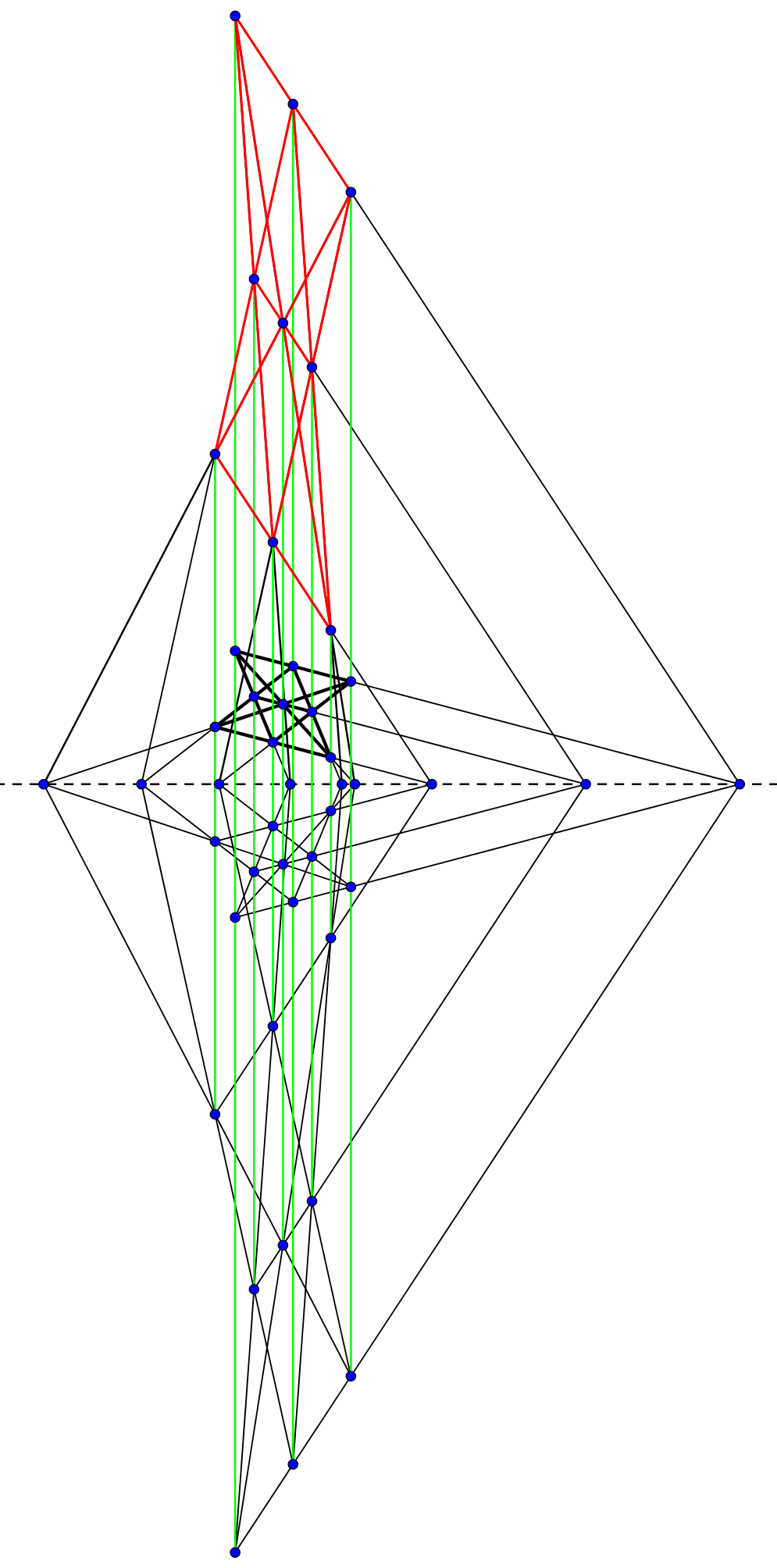}
\caption{Affine replication $\AR(9,4)$ applied to the $(9_3)$ Pappus configuration, which yields a $(45_4)$ configuration. 
The starting figure is indicated by thick segments, while the first image is highlighted by red segments. The axis $A$ is shown by a dashed line. 
The construction is chosen so as to exemplify that ordinary mirror reflection can also be used. Note that the resulting configuration contains a 
pencil of 9 parallel lines arising from the construction, shown in green.    
}
\label{fig:PappusExt}
\end{center}
\end{figure}

We may summarize the above discussion as follows:

\begin{lemma}\label{lemma:kp1}
If affine replication $\AR(m, k)$ is applied to any $(m_{k-1})$ configuration, the result is 
a $(((k+1)m)_k)$ configuration with a pencil of $m$ parallel lines.
\end{lemma}

\subsection{Affine Switch}

In our description of this construction, we are inspired by Gr\"unbaum~\cite[\S 3.3, pp. 177--180]{Gru2009b} but we have chosen 
a slightly different approach (in particular, we avoid using 3-space). At the same time, we generalize it from $(m_4)$ to $(m_k)$.

A sketch of the construction is as follows: Suppose that $\mc{C}$ is an $(m_{k})$ configuration that contains a pencil $\mc{P}$ of 
$p$ parallel lines in one direction, and a pencil $\mc{Q}$ of $q$ parallel lines in a second direction, where the two pencils share 
no common configuration points; we say that the pencils are \emph{independent}. For each subpencil $\mc S$ of $\mc P$  and 
$\mc{T}$ of $\mc{Q}$ containing $s$ parallel lines and $t$ parallel lines respectively, with $1 \leq s \leq p$ and $0 \leq t \leq q$, 
we form the subfiguration $\hat{\mc C}$ by deleting $\mc S$ and $\mc T$ from $\mc C$ (here we use the term \emph{subfiguration} 
in the sense of Gr\"unbaum~\cite{Gru2009b}). We then carefully construct $k-2$ affine images of $\hat{\mc C}$ in such a way that 
for each (deleted) line $\ell$ in $\mc S$ and for each point $P_{1}, P_{2}, \ldots P_{k}$ on $\ell$, the collection of lines through each 
$P_{i}$ and its images all intersect in a single point $Y_{\ell}$, and simultaneously, for each line $\ell'$ in $\mc T$ and for each point 
$Q_{1}, Q_{2}, \ldots Q_{k}$ on $\ell'$, the collection of lines through $Q_{i}$ and its images all intersect in a single point $X_{\ell'}$. 
Let $\mc D$ be the collection of all the undeleted points and lines of $\hat{\mc{C}}$ and its affine images and for each of the deleted 
$\ell$ and $\ell'$, the new lines through each point $P_{i}$ $Q_{i}$ and their images, the points $Y_{\ell}$, and the points $X_{\ell'}$; 
then $\mc{D}$ is a $( ((k-1)m + s+t)_{k})$ configuration. 

As a preparation, we need the following two propositions.

\begin{prop}\label{prop:pencil}
Let $\alpha$ be a (non-homothetic) affine transformation that is given by a diagonal matrix with respect to the standard basis.
Note that in this case $\alpha$ can be written as a (commuting) product of two orthogonal affinities whose axes coincide with the 
$x$- and $y$-axis, respectively:
$$
\begin{matr}{cc}
a&0\\0&b
\end{matr}
=
\begin{matr}{cc}
a&0\\0&1
\end{matr}
\begin{matr}{cc}
1&0\\0&b
\end{matr}
=
\begin{matr}{cc}
1&0\\0&b
\end{matr}%
\begin{matr}{cc}
a&0\\0&1
\end{matr}.
$$
Let $P_0(x_0,0),P_1(x_0, y_1),\dots,P_k(x_0, y_k)$ be a range of $k+1$ different points on a line which 
is perpendicular to the $x$-axis and intersects it in $P_0$. Then the $k$ lines connecting the pairs of points 
$(P_1, \alpha(P_1)),\dots, (P_k, \alpha(P_k))$ form a pencil with center $C_x$ such that $C_x$ lies on the 
$x$-axis, and its position depends only on $\alpha$ and $x_0$.

Likewise, let $Q_0(x_0,y_0),Q_1(x_1, y_0),\dots,Q_k(x_k, y_0)$ be a range of $k+1$ different points on a 
line which is perpendicular to the $y$-axis and intersects it in $Q_0$. Then the $k$ lines connecting the pairs 
of points $(Q_1, \alpha(Q_1)),\dots, (Q_k, \alpha(Q_k))$ form a pencil with center $C_y$ such that $C_y$ 
lies on the $y$-axis, and its position depends only on $\alpha$ and $y_0$.
\end{prop}
\begin{proof} An elementary calculation shows that 
$$
C_x=C_x\left(0,\displaystyle\frac{a-b}{b-1}\,x_0\right),
\text{\;resp.\;\,}
C_y=C_y\left(0,\displaystyle\frac{b-a}{a-1}\,y_0\right)
$$
is the common point of intersection of any two, hence of all the lines in question.
\end{proof}
\medskip

\begin{prop} \label{prop:series}
Let $h\ge3$ be a positive integer, and for each $j$ with $j=1,\dots,h-1$, let the affine transformation $\alpha_j$
be given by the matrix 
\begin{equation} \label{matrix}
M_j =
\begin{matr}{cc}
\displaystyle\frac{h-j}{h}&0\\0&\displaystyle\frac{h+j}{h}
\end{matr}.
\end{equation}
Then for any point $P$, the points $P, \alpha_1(P),\dots,\alpha_{h-1}(P)$ are collinear.
\end{prop}
\begin{proof}
Choose any $j'$ and $j''$, and form the difference matrices $M_{j'}-U$ and $M_{j''}-U$ with the unit matrix $U$.
Observe that these matrices are such that one is a scalar multiple of the other. Hence the vectors
$\overrightarrow{PP'}$ and $\overrightarrow{PP''}$ are parallel, where $P'=\alpha_{j'}(P)$ and 
$P''=\alpha_{j''}(P)$. This means that the points $P$, $P'$ and $P''$ lie on the same line.
\end{proof}
\medskip

\begin{figure}[h!]
\begin{center}
\includegraphics[width=.9\textwidth]{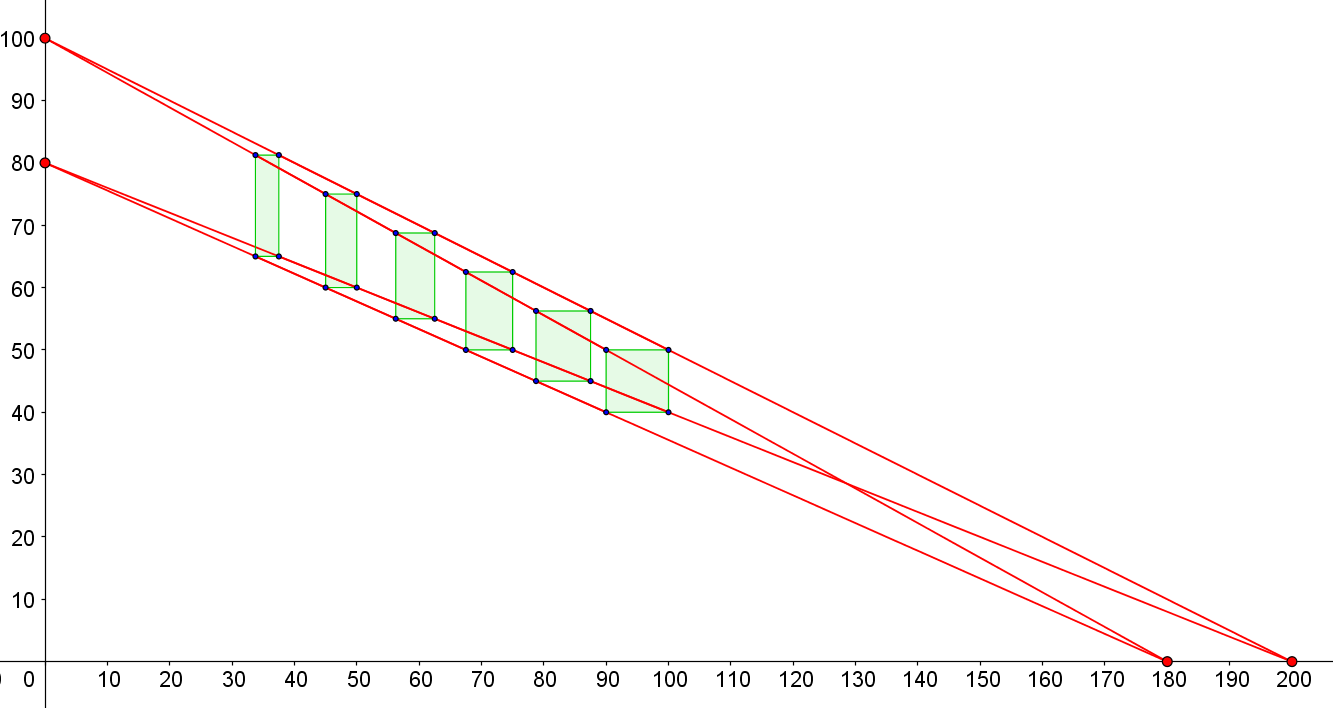}
\caption{Illustration for Propositions~\ref{prop:pencil} and~\ref{prop:series}. 
             Affine transformations with parameters $h=8$ and $j=1,\dots 5$ are applied on a square.}
\label{fig:series}
\end{center}
\end{figure}

Now we apply the following construction. Let $\mathcal C$ be an $(m_k)$ configuration such that it contains a pencil $\mc P$ of 
$p\ge1$ parallel lines and a pencil $\mc Q$ of $q\ge1$ parallel lines, too, such that these pencils are perpendicular to each other and are independent.
Note that any configuration containing independent pencils in two different directions can be 
converted by a suitable affine transformation to a configuration in which these pencils will be perpendicular to each other.

Choose a position of $\mathcal C$ (applying an affine transformation if necessary) such that these pencils are parallel to the 
$x$-axis and $y$-axis, respectively. 
\begin{enumerate}
\item
Remove lines $\ell_1,\dots, \ell_s$ $(s\le p)$ from the pencil $\mc P$ parallel to the $x$-axis and $\ell_{s+1},\dots, \ell_{s+t}$ $(0 \leq t\le q)$
from the pencil $\mc Q$ parallel to the $y$-axis. Let $\widehat{\mathcal {C}}$ denote the substructure of $\mathcal C$ obtained in this way. 
\item \label{item:j}
Let $h$ be a positive integer (say, some suitable multiple of $k$), and for each $j$, $j=1, \dots, k-2$, let $\alpha_j$ be an
affine transformation defined in Proposition~\ref{prop:series}. Form the images $\alpha_j(\widehat{\mathcal {C}})$ for 
all $j$ given here.
\item\label{item:ConnectingLine}
Let $P$ be a point of $\widehat{\mathcal {C}}$ that was incident to one of the lines $\ell_i$ removed from $\mathcal C$. 
Take the images $\alpha_j(P)$ for all $j$ given in (\ref{item:j}). By Proposition~\ref{prop:series}, all the $k-1$ points 
$P, \alpha_j(P)$ are collinear. Let $c_i(P)$ denote this line.
\item\label{item:centre}
Take all the configuration points on $\ell_i$ and repeat (\ref{item:ConnectingLine}) for each of them. By Proposition~\ref{prop:pencil},
the $k$-set of lines $\{c_i(P): P\in \ell_i\}$ form a pencil whose center lies on the $x$-axis or the $y$-axis according to which axis $\ell_i$ is 
perpendicular to. 
\item \label{item:NewElements}
Let $r=(s+t)\in\{1,2, \dots, p+q\}$ be the number of lines removed from the pencils of $\mathcal C$ in the initial step of our construction. 
Repeat (\ref{item:centre}) for all these lines. Eventually, we obtain $rk$ new lines and $r$ new points such that the set of the new
lines is partitioned into $r$ pencils, and  the new points are precisely the centers of these pencils (hence they lie on the coordinate 
axes). Observe that there are precisely $k$ lines passing through each of the new points, and likewise there are precisely $k$ points 
lying on each of the new lines.
\item
Putting everything together, we form a $(((k-1)m+r)_k)$ configuration, whose 
\begin{itemize}
\item
points come from the $(k-1)m$ points of the copies of $\widehat{\mathcal {C}}$, completed with the $r$ new points considered 
in (\ref {item:NewElements}).%
\item
lines come from the $(k-1)(m-r)$ lines of the copies of $\widehat{\mathcal{C}}$, completed with the $rk$
new lines considered in (\ref{item:NewElements}).
\end{itemize}

We use the notation $AS(m, k, r)$ to represent the $(((k-1)m+r)_k)$ configuration described above.
\end{enumerate}
%


Summarizing the discussion above, we conclude:


%
\begin{lemma}\label{lemma:km1pr}
Beginning with any $(m_k)$ configuration with independent pencils of $p\geq 0$ and $q \geq 1$ parallel lines, 
for each integer $r$ with $1 \leq r \leq p+q$, the affine switch construction produces an $(n_k)$ configuration, 
where $n = (k-1)m+r$.
\end{lemma}

Note that $p+q$ independent lines in an $(m_k)$ configuration covers $k(p+q) \leq m$ points. This gives an upper bound 
$p+q \leq m/k$, where the equality is attained only if $m$ divides $k$. 

In this paper we use the above Lemma \ref{lemma:km1pr} in connection with Lemma \ref{lemma:kp1} only for the case 
of a single pencil of parallel lines, such that $p = 0$. 

\begin{corollary}\label{lem:affineSwitch}
From any starting $(m_{k})$ configuration that has a pencil of $q$ parallel lines, we  apply a sequence of affine switches 
by removing $1, 2, \ldots, q$ lines in sequence, to construct a  sequence of  consecutive configurations 
\[  [ (((k-1)m+r)_{k})  ]_{r=1}^{q} = [  AS(m, k, r) ]_{r=1}^{q}. \]
\end{corollary}

This collection of consecutive configurations is represented by the notation $\AS(m, k, q)$. 
That is, $\AS(m,k,q) =  [  AS(m, k, r) ]_{r=1}^{q}.$

\begin{example}
Figure~\ref{fig:(k-1)m+} illustrates an application of this construction to the Pappus configuration $\mc{P}$ 
(cf.\ Figure~\ref{fig:pappus}).
Removing only one line from the horizontal pencil results in a $(19_3)$ configuration, shown in Figure \ref{fig:(k-1)m+}(a). 
Removing two or three lines results in a $(20_3)$ or $(21_3)$ configuration, respectively, shown in Figures \ref{fig:(k-1)m+}(b) 
and \ref{fig:(k-1)m+}(c). 
(Observe that since the Pappus configuration has 9 points, the maximal total number of lines in independent pencils is 3, 
since any three disjoint lines in the configuration contain all the points of the configuration.) Taken together the three 
configurations, we have: $[(19_{3}), (20_{3}), (21_{3})] = \AS(9,3,3)$. 
\end{example}

\begin{figure}[htbp]
\begin{center} \hskip -10pt 
\subfigure[A $(19_{3})$ configuration]
   {\includegraphics[width=0.3\textwidth]{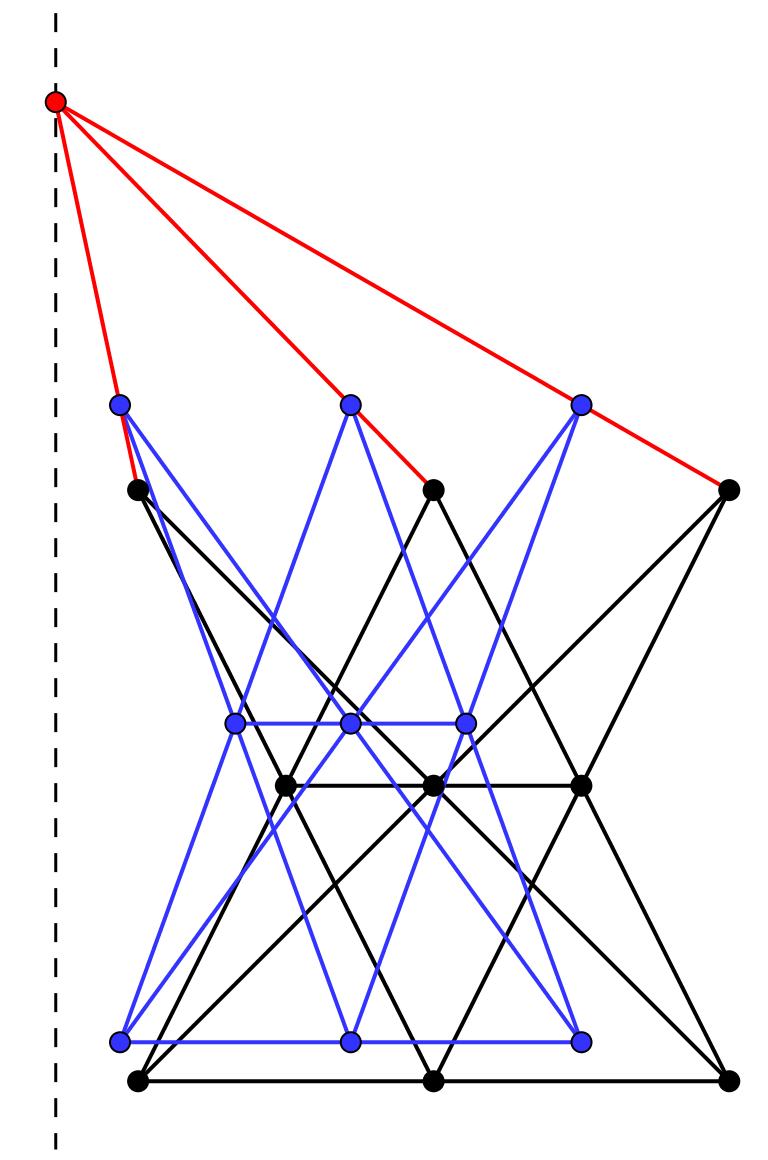}} \hskip 6pt
\subfigure[A $(20_{3})$ configuration]
   {\includegraphics[width=0.3\textwidth]{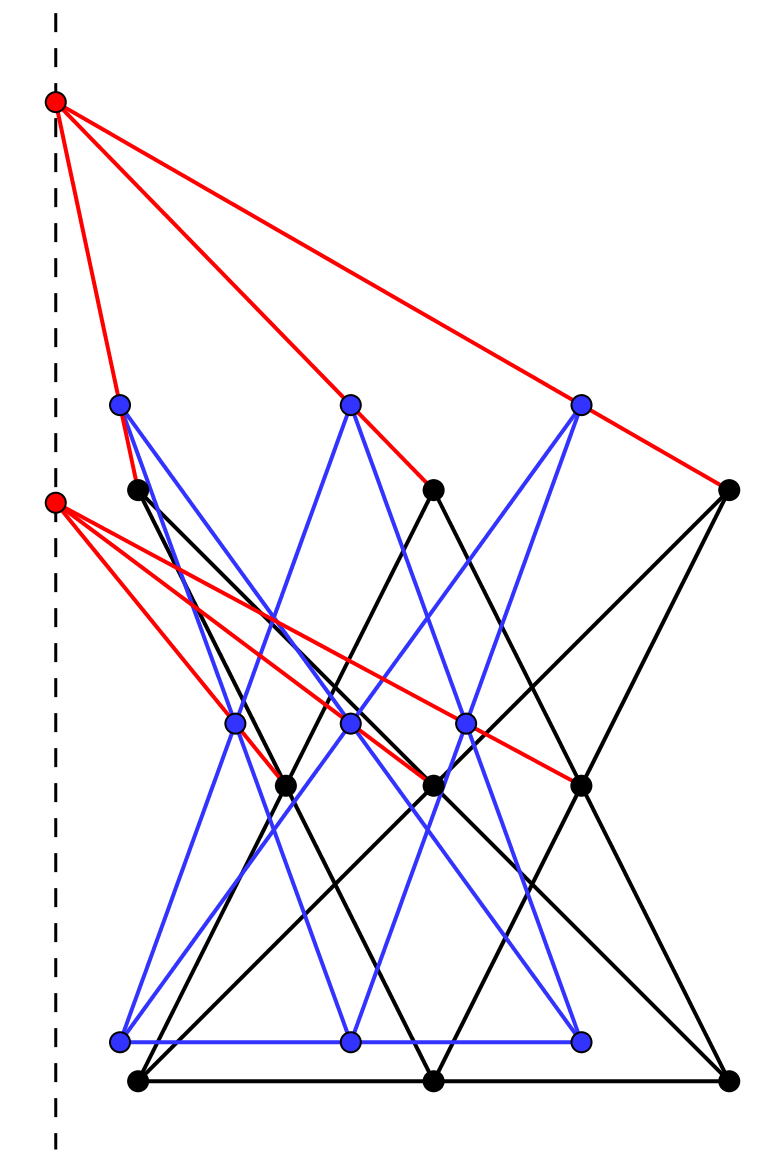}} \hskip 6pt
\subfigure[A $(21_{3})$ configuration]
   {\includegraphics[width=0.3\textwidth]{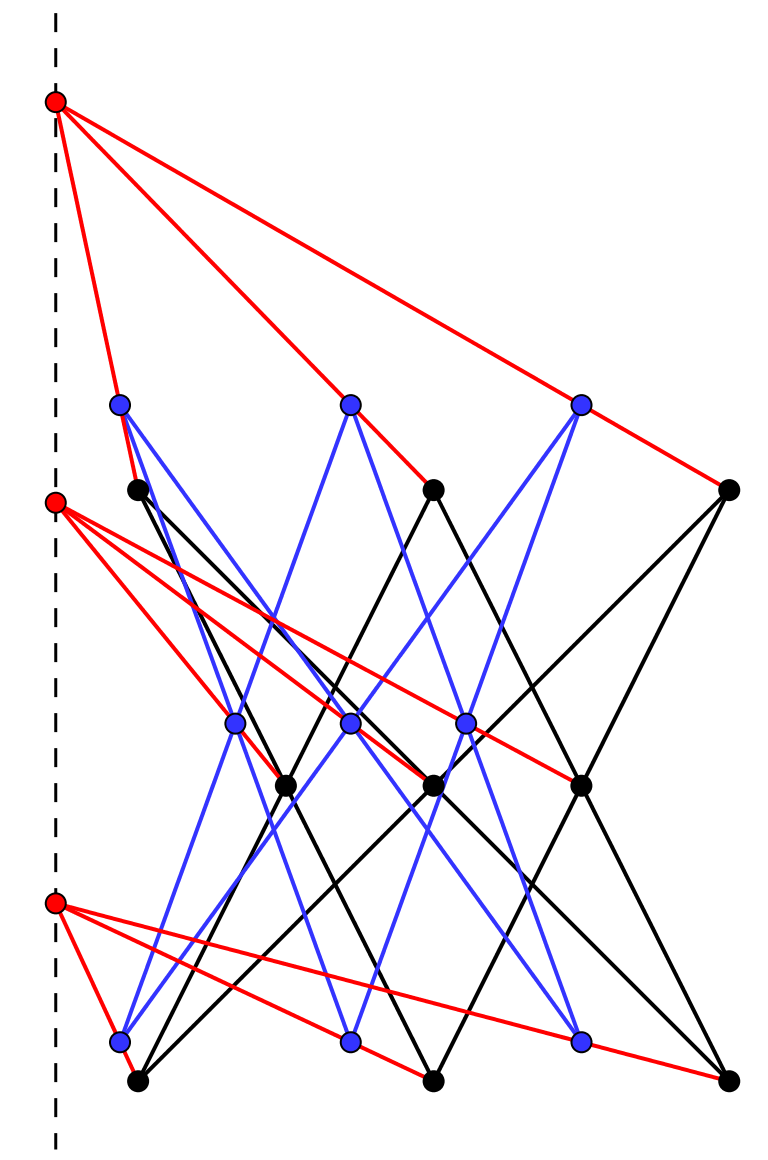}}
\caption{Configurations $(19_{3})$, $(20_{3})$, and $(21_{3})$, constructed from applying the affine switch construction to the realization 
of the Pappus configuration with a pencil of 3 parallel lines, shown in Figure \ref{fig:pappus}, by deleting one, two, or three lines respectively.
(The vertical axis of affinity, denoted by dashed line, does not belong to the configuration.)}
\label{fig:(k-1)m+}
\end{center}
\end{figure}

Since axial affinities play a crucial role in the constructions described above, we recall a basic property. 
The proof of the following proposition is constructive, hence it provides a simple tool for a basically synthetic 
approach to these constructions, which is especially useful when using dynamic geometry software to construct these configurations.

\begin{prop} \label{prop:determined}
An axial affinity $\alpha$ is determined by its axis and the pair of points $(P,P')$, where $P$ is any point not lying on the axis,
and $P'$ denotes the image of $P$, i.e.\ $P'=\alpha(P)$.
\end{prop}
\begin{proof}
In what follows, for any point $X$, we denote its image $\alpha(X)$ by $X'$. 
Let $Q$ be an arbitrary point not lying on the axis and different from $P$. 
Take the line $PQ$, and assume that it intersects the axis in a point $F$
(see Figure~\ref{fig:AffineConstr}a).
Thus $PQ=FP$. Take now the line $F'P'$, i.e., the image of $FP$. Since $F$ 
is a fixed point, i.e.\ $F'=F$, we have $F'P'=FP'$. This means that $Q'$ lies 
on $FP'$, i.e.\ $P'Q'=FP'$. To find $Q'$ on $FP'$, we use the basic property 
of axial affinities that for all points $X$ not lying on the axis, the lines $XX'$
are parallel with each other (we recall that the direction of these lines is called 
the \emph{direction} of the affinity). Accordingly, a line passing through $Q$
which is parallel with $PP'$ will intersect $FP'$ precisely in the desired point 
$Q'$.

\begin{figure}[h!]
\begin{center}
\subfigure[]
{\includegraphics[width=.5\textwidth]{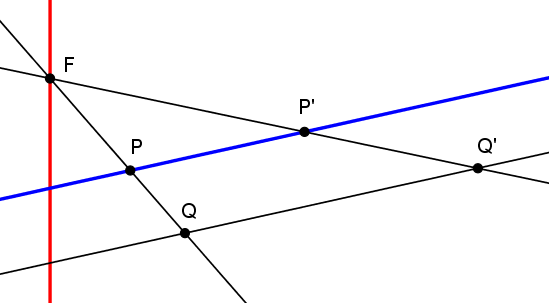}} \hskip 20pt
\subfigure[]
{\includegraphics[width=.35\textwidth]{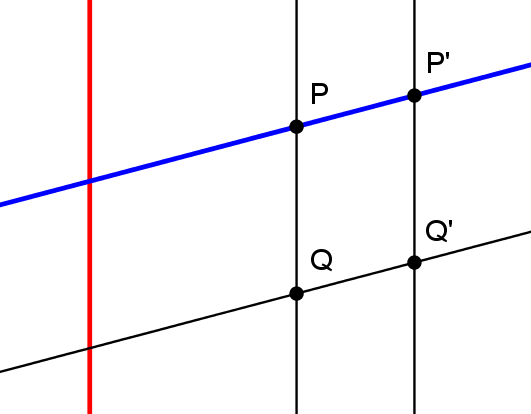}}
\caption{Construction of the image of a pont $Q$ under axial affinity; the axis is the vertical red line, the direction of affinity is given 
by the blue line. Here we use oblique affinity, but the construction given in the proof is the same in any other types of axial 
affinities.}
\label{fig:AffineConstr}
\end{center}
\end{figure}

On the other hand, if $PQ$ is parallel with the axis, then clearly so is $P'Q'$. 
In this case $Q'$ is obtaned as the fourth vertex of the parallelogram determined 
by $P'$, $P$ and $Q$ (see Figure~\ref{fig:AffineConstr}b).
\end{proof}

\begin{remark}
In using integer parameters $h$ and $j$ above, we followed Gr\"unbaum's original concept~\cite{Gru2009b} (as mentioned 
explicitly at the beginning of this subsection). However, the theory underlying Propositions~\ref{prop:pencil} and~\ref{prop:series} 
makes possible using continuous parameters as well, so that the procedure becomes in this way much more flexible. In what follows
we outline such a more general version, restricted to using only one pencil of lines to be deleted. 
\end{remark}

Start again with a configuration $\mathcal C$, and assume that the pencil $\mathcal P$ is in horizontal position; accordingly, the axis 
that we use is in vertical position (see e.g.\ Figure~\ref{fig:(k-1)m+}). Choose a line $\ell$ in $\mathcal P$, and a configuration point 
$P_0$ on $\ell$; then, remove $\ell$. $P_0$ will be the initial point of our construction (e.g., in Figure~\ref{fig:(k-1)m+} the ``north-west" 
(black) point of the starting configuration). Choose a point $C_{\ell}$ on the axis such that the line $C_{\ell}P_0$ is not perpendicular to 
the axis (in our example, this is the red point in Figure~\ref{fig:(k-1)m+}a).

Now let $t\in\mathbb R$ be our \emph{continuous parameter}. Take the point 
\begin{equation} \label{eq:AffCombin}
P=tC_{\ell}+(1-t)P_0;
\end{equation}
 thus $P$ is a point on the line $C_{\ell}P_0$, and as $t$ changes, $P$ slides along this line. Moreover, by Proposition~\ref{prop:determined}
we see that the pair of points $(P_0,P)$ determines two orthogonal affinities whose axes are perpendicular to each other. In particular, the
axes are precisely the coordinate axes. These affinities act simultaneously, i.e.\ $P_0$ is sent to $P$ by their (commuting) product. Using coordinates,
such as $P_0(x_0,y_0)$ and $P(x,y)$, we also see that the ratio of these affinities is $y/y_0$ (for that with horizontal axis), respectively $x/x_0$ 
(for that with vertical axis). (Note that these ratios, using the relation~(\ref{eq:AffCombin}), can also be expressed by the parameter $t$ and by 
the prescribed coordinates of $P_0$ and $C_{\ell}$. Furthermore, similarly, the matrix~(\ref{matrix}) above can also be parametrized by $t$; 
we omit the details.)

It is easily checked that both Proposition~\ref{prop:pencil} and Proposition~\ref{prop:series} remains valid with this continuous parameter $t$.
Hence, for any $P$, we can construct the corresponding affine image of $\mathcal C$ (or its substructures $\hat {\mathcal C}$ with lines of any 
number $r$ removed), together with the new lines (which are denoted by red in our example of Figure~\ref{fig:(k-1)m+}). In particular, in case 
of $k$-configurations, we need to choose altogether $k-2$ points on the line $C_{\ell}P_0$ (note the for $t=0$, the starting copy $\mathcal C$
returns, for $t=1$ the image of $\mathcal C$ collapses to a segment within the $y$-axis, and for a third value depending on the slope of $C_{\ell}P_0$,
it collapses to a segment within the $x$-axis; these cases thus are to be avoided).


\section{Proof of the Main Theorem}\label{sec:mainThms}


In this section we prove the main theorem of our paper.   For notational convenienence, given integers $a < b$, let $[a:b]$ denote 
the range $\{a, a+1, \ldots, b\}$. Similarly, for integer function $f(s)$ the range
$\{f(a), f(a+1), \ldots, f(b)\}$ will be denoted by $[f(s)]^b_{s=a}$. 
The crucial step in the proof will be provided by the following Lemma.

\begin{lemma}\label{mainLemma}
Assume that for some $k \geq 3$, $N_{k-1}$ exists and that $\bar{N}_{k-1}$ is any known upper bound for it. 
Then $N_k$ exists and:
$\bar{N}_k =  (k^2-1)\max(\bar{N}_{k-1},k^2-2)$ is an upper bound for it. Moreover, if we have two upper bounds, say $\bar{N}_{k-1} < \tilde{N}_{k-1}$
for $N_{k-1}$, the better one will produce a better upper bound for $N_{k}$.
\end{lemma}

This Lemma will be proven with the tools from previous section by  applying affine replication and affine switch.
More precisely, Lemma \ref{lemma:kp1} and Corollary  \ref{lem:affineSwitch} will be used.


\begin{proof}[Proof of Lemma \ref{mainLemma}]

Let $\bar{N}_{k-1}$ denote any known upper bound for $N_{k-1}$. By definition, the  sequence 
of consecutive numbers
\begin{equation} \label{eq:sequence-1}
a=\bar{N}_{k-1}, a+1, \ldots, a+s, \ldots
\end{equation}
are all $(k-1)$-realizable; in other words, for each $s$, $s=0,1,\dots,$ there exists a geometric $((a+s)_{k-1})$ 
configuration (recall the definition of realizability, given in the Introduction). Apply affine replication to these configurations; 
by Lemma~\ref{lemma:kp1}, the  sequence of numbers:
\begin{equation} \label{eq:sequence-2}
(k+1)a, (k+1)(a+1),\ldots, (k+1)(a+s), \dots
\end{equation}
are all $k$-realizable. Note that this is an arithmetic sequence with difference $(k+1)$.
Furthermore, observe that for each $X \geq a$, the geometric $k$-configuration realizing the number $(k+1)X$ 
that was produced by affine replication has $X$ new parallel lines. 
Hence, we can apply a sequence of affine switch constructions to each of these configurations $((k+1)X_{k})$.
By Corollary \ref{lem:affineSwitch}, the sequences $\AS((k+1)X, k, X)$ of configurations is produced.
It follows that the sequences of numbers
\begin{multline}\label{eq:sequence-3}
[(k-1) (k+1) a+1:(k-1) (k+1)a+a], \\
[(k-1) (k+1)(a+1)+1:(k-1) (k+1) (a+1)+(a+1))], \\
[(k-1) (k+1)(a+2)+1:(k-1) (k+1) (a+2)+(a+2))],\ldots
\end{multline}
are all $k$-realizable.

Observe that from the initial outputs of affine replication, $n = X(k+1)$ is realizable as long as $X \geq \bar{N}_{k-1}$. 
Thus, every ``band'' of consecutive configurations produced by affine switches can be extended back one step, so there 
exists a band of consecutive $k$-configurations 
\[[(k-1)(k+1)X:(k-1) (k+1)X + X)] \]
for each initial configuration $(X_{k-1})$. 
Another way to say this is that we can fill a hole of size 1 between the bands of configurations listed in equation \eqref{eq:sequence-3} 
using the output of the initial affine replications, listed in equation \eqref{eq:sequence-2}.

To determine when we have either adjacent or overlapping bands, then, it suffices to determine when the last element of one band is 
adjacent to the first element of the next band; that is, when 
\[(k-1) (k+1) X+X + 1\geq (k-1) (k+1) (X+1).\]
It follows easily that $X \geq k^{2}-2$.

Hence, as long as we are guaranteed that a sequence of consecutive configurations $(q_{k-1})$, $((q+1)_{k-1}), \ldots$ 
exists, it follows that we are guaranteed the existence of consecutive $k$-config\-u\-ra\-tions $Q_{k}, (Q+1)_{k}, \ldots,$ 
where $Q = (k^{2}-1)(k^{2}-2)$. 
However, since we do not know whether that consecutive sequence exists, in the (extremely common) case 
where $\bar{N}_{k-1} > (k^{2}-1)(k^{2}-2)$, the best that we can do is to conclude that
\[ N_{k} \leq (k^{2}-1) \max\{ \bar{N}_{k-1}, k^{2} - 2\}.\]
\end{proof}

This result gives rise to an elementary proof by induction for the main theorem.

\begin{proof}[Proof of Theorem \ref{mainTheorem}] 
Let $s = 2$.  The number $N_{s} = N_2 = 3$ exists. This is the basis of induction. Now, let $s = k -1$. By assumption, $N_{k-1}$ exists 
and some upper bound $\bar{N}_{k-1}$ is known.  By Lemma \ref{mainLemma}, $\bar{N}_k =  (k^2-1)\max(\hat{N}_{k-1},k^2-2)$ is 
an upper bound for $N_k$ . Therefore $N_k$ exists and the induction step is proven.
\end{proof}

Recall that we let $N^{R}_{k}$ denote the best known upper bound for $N_{k}$. The same type of result follows if we start 
with the best known upper bound $N^R_s$ for some $s \geq 2$. However, the specific numbers for upper bounds depend on 
our starting condition. Table \ref{tab:theorem1bounds} shows the difference if we start with $s = 2,3,4$. The reason we are 
using only these three values for $s$ follows from the fact that only $N^R_s, 2 \leq s \leq 4$ have been known so far.

\begin{table}[htp]
\caption{Bounds on $N_{k}$ from iterative applications of Lemma \ref{mainLemma}. Different bounds are produced if the iteration 
is started with $N^R_{2} = N_{2} = 3, N^R_{3}= N_{3} = 9$ or with $N^R_{4} = 24$. Boldface numbers give best bounds using  
this method and current knowledge.}
\begin{center}

\begin{tabular}{c|r|r|r|r}

$k$ & $\bar{N}_{k}$ with $N^R_{2} = 3$ &$\bar{N}_{k}$ with $N^R_{3} = 9$ & $\bar{N}_{k}$ with $N^R_{4} = 24$ &$N^R_{k}$ \\[3 pt] \hline
$2$ &{\bf 3} & - & - & {\bf 3}\\
$3$ &56& {\bf 9} & - & {\bf 9}\\
$4$ &840& 210 & {\bf 24} & {\bf 24} \\
$5$ &\numprint{20160}&\numprint{5040} & {\bf 576}& {\bf 576}\\
$6$ &\numprint{705600} &\numprint{176400}& {\bf \numprint{20160}}& {\bf \numprint{20160}}\\
$7$ &\numprint{33868800} &\numprint{8467200}& {\bf \numprint{967680}}& {\bf \numprint{967680}}\\
$8$ & \numprint{2133734400}&\numprint{533433600}& {\bf \numprint{60963840}}& {\bf \numprint{60963840}}\\
$9$ &\numprint{170698752000} &\numprint{42674688000}& {\bf \numprint{4877107200}}&{\bf \numprint{4877107200}}\\
$10$ &\numprint{16899176448000} &\numprint{4224794112000}&{\bf  \numprint{482833612800}}&{\bf \numprint{482833612800}}\\

\end{tabular}

\end{center}
\label{tab:theorem1bounds}
\end{table}%
%
The rightmost column of Table \ref{tab:theorem1bounds}  summarises the information given in other columns by computing the minimum in 
each row and thereby gives the best bounds that are available using previous knowledge and direct applications of  Lemma \ref{mainLemma}.

If new knowledge about best current values of $N^R_k$ for small values of $k$ becomes available, we may use similar applications of  
Lemma \ref{mainLemma} to improve the bounds of the last column. Since, the values for $k=2$ and $k=3$ are optimal, the first candidate 
for improvement is $k = 4$. A natural question is what happens if someone finds a geometric $(23_4)$ configuration. In this case Lemma 
\ref{mainLemma} would give us for $k = 5$ the bound $(k^2-1)\max(N^{R}_{k-1},k^2-2) = (5^2-1)\max(20,5^2-2) = 24 \times 23 = 552$, 
an improvement over 576. An alternative feasible attempt to improve the bounds would be to use other methods in the spirit of Gr\"unbaum 
calculus to improve the current bound 576 for $k=5$. However, there is another approach that can improve the numbers even without 
introducing new methods. It is presented in the next section.


\section{Improving the bounds}


Recall that $N^{R}_{3} = N_{3} = 9$, and $N^{R}_{4} = N_{4} = 21$ or $24$, according to whether or not a 
$(23_{4})$ configuration exists. If we apply the procedure in Lemma \ref{mainLemma} using as input information 
$N_{3} = N^{R}_{3} = 9$ (that is, beginning with a sequence of $3$-configurations $(9_{3}), (10_{3}), (11_{3}) \ldots$), 
Lemma \ref{mainLemma} says that 
\[ N_{k}\leq (k^{2}-1)\max\{N^{R}_{k-1},k^{2} - 2\} \implies N_{4} \leq (15)\max\{9, 14\} = 210. \]

However, we know observationally that $N_{4} = 21$ or $24$. Thus, we expect that Theorem \ref{mainTheorem} 
is likely to give us significant overestimates on a bound for $N_{k}$ for larger $k$.

For $k = 5$, the best we can do at this step with these constructions is the bound given by Lemma \ref{mainLemma}, 
beginning with the consecutive sequence of $4$-configurations $((24_{4}), (25_{4}), (26_{4}), \ldots)$.
In this case, Lemma \ref{mainLemma} predicts that 
$N_{5} \leq (24)\max(24, 23) = 576.$
In a subsequent paper, we will show that this bound can be significantly decreased by incorporating other Gr\"unbaum-calculus-type constructions 
and several ad hoc geometric constructions for 5-configurations.

However, we  significantly decrease the bound on $N_{k}$ for $k \geq 6$ by refining the construction sequence given in Lemma \ref{mainLemma}:
instead of beginning with $N^{R}_{k-1}$ determined by iterative applications of the sequence in Lemma \ref{mainLemma}, we consider all possible 
sequences determined by applying a series of affine replications, followed by a final affine switch.

First we introduce  a function $N(k,t,a,d)$ with positive integer parameters $k,t,a,d$ and $t < k$.
Define for $t < k-1$: 
\[N(k, t, a, d):= (k^{2} - 1)\left(\frac{k!}{(t+1)! }\right) \max\left\{  a, (k^{2}-1)d \right\}, \] 

and for $t = k-1$: 
\[N(k, k-1, a, d):= (k^{2} - 1) \max\left\{ a, (k^{2}-1)d - 1 \right\}. \]

This value $N(k, t, a, d)$ is precisely the smallest $n$ after which we are guaranteed there exists a sequence of consecutive $k$-configurations 
produced by starting with an initial sequence of $t$-configurations $a, a+d, ..., $ and sequentially applying affine replications followed by a final 
affine switch as described above. 
 
The following Lemma gives us a quite general and powerful tool for bound improvements without making any changes in constructions.

\begin{lemma}\label{lemma:main2} Let $t \geq 2$ be an integer and let $a, a+d, a+2d, \ldots$ be an arithmetic sequence with integer initial 
term $a$ and integer difference $d$ such that for each $s = 0,1, \ldots$ geometric configurations $((a+sd)_t)$ exist.  
Then for any $k > t$  the value $ N(k, t, a, d)$  defined above is an upper bound for  $N_k$; i.e.,
$N(k, t, a, d) \geq N_k$. 
\end{lemma}


\begin{proof}[Proof of Lemma \ref{lemma:main2}]
Beginning with an arithmetic sequence of $t$-configurations, we construct a consecutive sequence of $k$-configurations 
by iteratively applying a sequence of affine replications to go from $t$-configurations to $(k-1)$-configurations; a final affine 
replication to go from  $(k-1)$-configurations to $k$-configurations with a known number of lines in a parallel pencil; and finish 
by applying affine switch on that final sequence of $k$-configurations to produce bands of consecutive configurations. We then
analyze at what point we are guaranteed that the bands either are adjacent or overlap.

Specifically, starting with a sequence of $t$-realizable numbers $a, a+d, a+2d, \ldots$ we successively apply $k-t$ affine replications 
to the corresponding sequence of configurations to form sequences of $s$-realizable numbers for $t \leq s \leq k$:
\begin{align}
a, a+d, a+2d, \ldots  &\xrightarrow[(t+1)\text{-cfgs}]{ \AR(\cdot, t+1)} (t+2)a, (t+2)(a+d), (t+2)(a+2d), \ldots  \nonumber
\\
 &\xrightarrow[(t+2)\text{-cfgs}]{ \AR( \cdot, t+2)} (t+3)(t+2)a, (t+3)(t+2)(a+d), (t+3)(t+2)(a+2d), \ldots \nonumber \\
&\vdots \nonumber\\
&\xrightarrow[k\text{-cfgs}]{ \AR( \cdot, k)} \frac{(k+1)!}{(t+1)!}a, \frac{(k+1)!}{(t+1)!}(a+d), \frac{(k+1)!}{(t+1)!}(a+2d), \ldots  \label{eq:sequence-4}
\end{align}
By Lemma \ref{lemma:kp1}, each of the $k$-configurations corresponding to the realizable numbers in equation \eqref{eq:sequence-4} 
produced from a starting configuration $X$ has a pencil of  $\frac{k!}{(t+1)!}X$ parallel lines. To those configurations we apply 
the affine switch operation:
\begin{multline}
\frac{(k+1)!}{(t+1)!}a, \frac{(k+1)!}{(t+1)!}(a+d), \frac{(k+1)!}{(t+1)!}(a+2d), \ldots \\ 
\xrightarrow[k\text{-cfgs}]{ \AS(\cdot, k, \cdot)} \left[(k-1)\frac{(k+1)!}{(t+1)!}a+1: (k-1)\frac{(k+1)!}{(t+1)!}a+\frac{k!}{(t+1)!}q\right], \\ 
\left[(k-1)\frac{(k+1)!}{(t+1)!}(a+d)+1: (k-1)\frac{(k+1)!}{(t+1)!}(a+d)+\frac{k!}{(t+1)!}(a+d)\right], \ldots \label{eq:finalBands}
\end{multline}

As in the proof of Theorem \ref{mainTheorem}, observe that the $(n_{k})$ configurations described in \eqref{eq:sequence-4} 
all have $n$ a multiple of $\frac{(k+1)!}{(t+1)!}$. That is, any $n$ divisible by $\frac{(k+1)!}{(t+1)!}$ is $k$-realizable as long 
as when $n = \frac{(k+1)!}{(t+1)!}X$, $X$ is larger than $N^{R}_{t}$. We thus can extend our band of consecutive realizable 
configurations back one step, to be of the form
\[ \left[(k-1)\frac{(k+1)!}{(t+1)!}X: (k-1)\frac{(k+1)!}{(t+1)!}X+\frac{k!}{(t+1)!}X\right]\]
for a starting $t$-realizable number $X$. 

Successive bands of this form are guaranteed to either exactly meet or to overlap when the end of one band, plus one, 
equals or is greater to the beginning of the next, that is, when
\begin{align}
(k-1)\frac{(k+1)!}{(t+1)!}X+\frac{k!}{(t+1)!}X +1 &\geq (k-1)\frac{(k+1)!}{(t+1)!}(X+d) \implies \nonumber \\ 
X &\geq (k^{2}-1)d - \frac{(t+1)!}{k!}. \label{stupidInequality}
\end{align}

When $t = k-1$, $\frac{(t+1)!}{k!}  = 1$, while when $t <k-1$, $\frac{(t+1)!}{k!}  < 1$, and moreover, inequality \eqref{stupidInequality} 
holds as long as $X$ is greater than the bound on $t$-realizable configurations.
\end{proof}


We refine and improve the upper bounds of Table \ref{tab:theorem1bounds} with Theorem \ref{thm:main2}. 
This proof proceeds by showing, given a starting arithmetic sequence of consecutive $t$-configurations, 
a construction method for producing a sequence of consecutive $k$-configurations.

\begin{theorem}\label{thm:main2} Recursively define \[\hat{N}_{k} = (k^{2}-1)\min_{3 \leq t < k}\{ N(k, t, \hat{N}_{t}, 1)\}\]
with $\hat{N}_{3} = N_{3} = 9$ and $\hat{N}_{4}= N^{R}_{4}= 24$.
Then $\hat{N}_{k}$ is an upper bound for $N_k$. 
\end{theorem}

\begin{proof}
Observe that by unwinding definitions, \[\hat{N}_{k} = (k^{2}-1) \min_{3 \leq t \leq k-1} \left\{ \frac{k!}{(t+1)!}\max\left\{\hat{N}_{t}, k^{2}-1\right\}\right\}.\]

By construction, since for each $\hat{N}_{k}$ we have shown there exists consecutive $k$-configurations for each $n \geq \hat{N}_{k}$, it follows that
$N_{k} \leq \hat{N}_{k}$, and the result follows.
\end{proof}

Applying Theorem \ref{thm:main2} results in the bounds for $N_{k}$ are shown in Table \ref{tab:thm2bounds-large}.

\begin{table}[htp]
\caption{Bounds on $N_{k}$ produced from Theorem \ref{thm:main2}. The values for $N^{R}_{k}$ given in this table agree with
the record values listed in Table \ref{tab:theorem1bounds} for all $k \leq 5$ (boldface), and are strictly better for  $k \geq 6$.}
\begin{center}

\begin{tabular}{c | r | l | l}
$k$ & $\hat{N}_{k} = N^{R}_{k}$ & formula & initial sequence \\ \hline
4 & {\bf 24} & - & -\\
5 & ${\bf 576}$ & $(5^{2}-1)^{2}$& $t = 4$\\
6 & $\numprint{7350}$ & $6(6^{2}-1)^{2}$ & $t= 4$\\
7 & $\numprint{96768}$ & $7\cdot6 \cdot (7^{2}-1)^{2}$ & $t = 4$ \\
8 & $\numprint{1333584}$ & $\frac{8!}{5!} (8^{2}-1)^{2}$ & $t = 4$ \\
9 & $\numprint{19353600}$ & $\frac{9!}{5!} (9^{2}-1)^{2}$ & $t = 4$ \\
10 & $\numprint{287400960}$ & $\frac{10!}{6!}\cdot \mathbf{ 576 }\cdot (10^{2}-1)$ & $\mathbf{t = 5}$ \\ 
11 & $\numprint{3832012800}$ & $\frac{11!}{6!}\cdot 576 \cdot (11^{2}-1)$ & $t = 5$ \\
$\vdots$ & & &  \\
24 & $\approx 2.85 \times 10^{26}$ & $\frac{24!}{6!}\cdot 576 \cdot (24^{2}-1)$ &  $t = 5$\\
25 & $\approx 8.39 \times 10^{27}$ & $\frac{25!}{6!}\cdot \mathbf{(25^{2}-1)^{2}}$ &  $t = 5$\\
26 & $\approx 8.02 \times 10^{30}$ &  $\frac{26!}{6!}\cdot (26^{2}-1)^{2}$ & $t = 5$\\
$\vdots$ & & & \\
32 & $\approx 3.82 \times 10^{38}$ & $\frac{32!}{6!}\cdot (32^{2}-1)^{2}$ & $t = 5$ \\
33 & $\approx 1.38 \times 10^{40}$ & $\frac{33!}{7!}\cdot \mathbf{7350} \cdot (33^{2}-1)$ & $\mathbf{t = 6}$ \\
$\vdots$ &&\\
85 & $\approx 2.97 \times 10^{132}$ &$\frac{85!}{7!}\cdot \mathbf{7350} \cdot (85^{2}-1)$& $t = 6$\\
86 & $\approx 2.63 \times 10^{134}$ &$\frac{86!}{7!} \cdot\mathbf{(86^{2}-1)^{2}}$& $t = 6$\\
$\vdots$ &&& \\
109&  $\approx 4.04 \times 10^{180}$&
$\frac{109!}{7!} (109^{2}-1)^{2}$
 &$t = 6$\\
110 & $\approx 4.61 \times 10^{182}$ & $\frac{110!}{8!}\cdot \frac{7!}{5!}\cdot (7^{2}-1)^{2} \cdot (110^{2}-1)$ & $\mathbf{t = 7}$
\end{tabular}
\end{center}
\label{tab:thm2bounds-large}
\end{table}%

There are some interesting things to notice about the bounds from Theorem \ref{thm:main2} shown in Table \ref{tab:thm2bounds-large}. 
First, note that $t = 3$ is never used in determining $\hat{N}_{k}$. Second, for example,
 the bound $\hat{N}_{10}$ uses an initial sequence of $5$-configurations, rather than starting with $4$-configurations. 
To understand why, observe that

\begin{align*}\hat{N}_{10} &= (k^{2}-1) \min_{3 \leq t \leq 9} \{ N(k, t, \hat{N}_{t}, 1)\}\\
& = 99 \min\biggl\{\frac{10!}{4!} \max\{\hat{N}_{3} = 9, 99\}, \frac{10!}{5!} \max\{\hat{N}_{4} = 24, 99\}, \frac{10!}{6!} \max\{\hat{N}_{5} = 576, 99\}, \\
& \phantom{===}\;\,\, \qquad \frac{10!}{7!} \max\{\hat{N}_{6}=7350, 99\}, \ldots, 
\frac{10!}{10!} \max\{\hat{N}_{9}, 99\}\biggr\}\\
&= 99 \min\left\{\frac{10!}{4!} 99, \frac{10!}{5!} 99, \frac{10!}{6!} 576, \frac{10!}{7!} \hat{N}_{6}, \ldots,  \hat{N}_{9}\right\}
\end{align*}

Since $ 6 \cdot 99 > 576$ (and the values $\hat{N}_{t}$ for $ 6\leq t \leq 9$ much larger than either), the minimum of that list is 
actually $\frac{10!}{6!}576$, and the computation for $\hat{N}_{10}$ starts with the sequence of consecutive $5$-configurations 
$(576_{5}), (577_{5}), \ldots$ 
rather than with $(24_{4}), (25_{4}), \ldots$.
%
%
Sequences with $t = 5$ begin to dominate when $6(k^{2}-1) > 576 = (5^{2}-1)^{2}$; that is, when $k \geq \lceil\sqrt{97} \rceil = 10$. 
Sequences with $t = 6$ begin to dominate when $7(k^{2}-1) > 6(6^{2}-1)^{2} = 7350$, or $k \geq \left\lceil\sqrt{1051}\right\rceil = 33$.
Sequences with $t = 7$ will dominate when $8(k^{2}-1) >7\cdot6 \cdot (7^{2}-1)^{2}$, that is $k \geq \lceil\sqrt{12097}\rceil = 110$. 
However, note that these bounds are absurdly large;  $\hat{N}_{110} \approx 4.6 \times 10^{182}$.

In addition, observe that since $k = 25$ is the smallest positive integer satisfying $k^{2}-1>576$, the bounds for $\hat{N}_{25}$ 
use the $25^{2}-1$ choice rather than $\hat{N}_{5}$ in taking the maximum, even though both $\hat{N}_{24}$ and $\hat{N}_{25}$ 
are starting with the same initial sequence of $5$-configurations, and there is a similar transition again at $k = 86$, when the function is using $6$-configurations to produce the maximum. At this position, since $85^{2} - 1 = 7224$ and $86^{2}-1 = 7395$, $\hat{N}_{85}$ uses $\hat{N}_{6} = 7350$, but $\hat{N}_{86}$ transitions to using $86^{2} - 1$ to compute the maximum.  

\section{Future work}


With better bounds  $N^R_{t}$ developed experimentally for small values of $t$, in the same way that $N^R_{4} = 24$ 
has been determined experimentally, we anticipate significantly better bounds $N^{R}_{k}$, for $k > t$, without changing 
the methods for obtaining the bounds. 

One obvious approach is to improve the bookkeeping even further. For instance, in Theorem \ref{thm:main2} we only used 
arithmetic sequences with $d = 1$ in $N(k,t,a,d)$ and ignoring any existing configuration $(m_t)$ for $m < N_t$. In particular, 
for $t = 4$, we could have used $N(k,4,18,2)$ since $18,20,22,24, \ldots $ form an arithmetic sequence of $4$-realizable numbers. 
Our experiments indicate that this particular sequence has no impact in improving the bounds. However, by carefully keeping track 
of the existing $t$-configurations below $N^R_t$, other more productive arithmetic sequences may appear.  

Another approach is to sharpen the bounds for $N_k$, for general $k$. This can be achieved, for instance, by generalizing some other 
``Gr\"unbaum Calculus'' operations, which we plan for a subsequent paper. We also plan to apply several ad hoc constructions for 
$5$- and $6$-configurations to further sharpen the bound for $N_{5}$ and $N_{6}$, which will, in turn,  lead to significantly better 
bounds for $N_{k}$ for higher values of $k$. However, based on the work involved in bounding $N_{4}$ and the fact that $N_{4}$ 
is not currently known (and on how hard it was to show the nonexistence of an $(19_{4})$ configuration), we anticipate that even 
determining $N_{5}$ exactly is an extremely challenging problem.

Finally, very little is known about existence results on \emph{unbalanced} configurations, that is, configurations $(p_{q}, n_{k})$ 
where $q \neq k$. While some examples and families are known, it would be interesting to know any bounds or general results on 
the existence of such configurations.

\section*{Acknowledgements}
G\'abor G\'evay's research is supported by the Hungarian National Research, Development and Innovation Office, OTKA grant No.\ SNN 132625.
Toma\v{z} Pisanski's research is supported in part by the Slovenian Research Agency (research program P1-0294 and research projects 
N1-0032, J1-9187, J1-1690, N1-0140, J1-2481), and in part by H2020 Teaming InnoRenew CoE. 

\bibliographystyle{plain}
\bibliography{Bounds_2020}

\begin{thebibliography}{10}

\bibitem{BokPil2015}
J{\"u}rgen Bokowski and Vincent Pilaud.
\newblock On topological and geometric $(19_4)$ configurations.
\newblock {\em European J.\ Combin.}, 50:4--17, 2015.

\bibitem{BokPil2016}
J{\"u}rgen Bokowski and Vincent Pilaud.
\newblock Quasi-configurations: building blocks for point-line configurations.
\newblock {\em Ars Math.\ Contemp.}, 10(1):99--112, 2016.

\bibitem{BokSch2013}
J{\"u}rgen Bokowski and Lars Schewe.
\newblock On the finite set of missing geometric configurations $(n_4)$.
\newblock {\em Comput.\ Geom.}, 46(5):532--540, 2013.

\bibitem{Cox1969}
Harold Scott~MacDonald Coxeter.
\newblock {\em Introduction to {G}eometry}.
\newblock Wiley, New York, 2nd ed. edition, 1969.

\bibitem{Cun2018}
Michael Cuntz.
\newblock $(22_4)$ and $(26_4)$ configurations of lines.
\newblock {\em Ars Math.\ Contemp.}, 14:157--163, 2018.

\bibitem{Gru2000}
Branko Gr{\"{u}}nbaum.
\newblock Connected {$(n_4)$} configurations exist for almost all {$n$}.
\newblock {\em Geombinatorics}, 10(1):24--29, 2000.

\bibitem{Gru2000b}
Branko Gr{\"{u}}nbaum.
\newblock Which {$(n_4)$} configurations exist?
\newblock {\em Geombinatorics}, 9(4):164--169, 2000.

\bibitem{Gru2002}
Branko Gr{\"{u}}nbaum.
\newblock Connected {$(n_4)$} configurations exist for almost all {$n$}---an
  update.
\newblock {\em Geombinatorics}, 12(1):15--23, 2002.

\bibitem{Gru2006}
Branko Gr{\"{u}}nbaum.
\newblock Connected {$(n_4)$} configurations exist for almost all
  {$n$}---second update.
\newblock {\em Geombinatorics}, 16(2):254--261, 2006.

\bibitem{Gru2008a}
Branko Gr{\"u}nbaum.
\newblock Musings on an example of {D}anzer's.
\newblock {\em European J.\ Combin.}, 29(8):1910--1918, 2008.

\bibitem{Gru2009b}
Branko Gr{\"u}nbaum.
\newblock {\em Configurations of {P}oints and {L}ines}, volume 103 of {\em
  Graduate Studies in Mathematics}.
\newblock American Mathematical Society, Providence, RI, 2009.

\bibitem{GruRig1990}
Branko Gr{\"u}nbaum and John~F.\ Rigby.
\newblock The real configuration $(21_4)$.
\newblock {\em J.\ London Math.\ Soc.}, 41:336--346, 1990.

\bibitem{PisSer2013}
Toma{\v z} Pisanski and Brigitte Servatius.
\newblock {\em Configurations from a {G}raphical {V}iewpoint}.
\newblock Birkh{\"a}user Advanced Texts. Birkh{\"a}user, New York, 2013.

\end{thebibliography}
\end{document}